 \documentclass[draft]{article}

\usepackage{amsmath,amsfonts,amsthm,amssymb,amscd,color}
\renewcommand{\Im}{\mathop{\rm Im}\nolimits}

\theoremstyle{plain} \newtheorem{theorem}{Theorem}[section]
\newtheorem{lemma}[theorem]{Lemma}
\newtheorem{assumption}[theorem]{Assumption}
\newtheorem{proposition}[theorem]{Proposition}
 \theoremstyle{definition}
\newtheorem{definition}[theorem]{Definition} \theoremstyle{remark}
\newtheorem{remark}[theorem]{Remark}

\newcommand{\R}{{\mathbb R}}

\newcommand{\Z}{{\mathbb Z}}

\newcommand{\N}{{\mathbb N}}

\def\im{{\rm i}}

\newcommand{\C}{\mathbb{C}}

\def\({\left(}
\def\){\right)}
\def\<{\left\langle}
\def\>{\right\rangle}

\numberwithin{equation}{section}

\begin{document}

\title{A note on small periodic solutions of discrete nonlinear Klein-Gordon equations}

\author {Masaya Maeda}

\maketitle

\begin{abstract}
In this note, we consider discrete nonlinear Klein-Gordon equations with potential.
By the pioneering work of Sigal, it is known that for the "continuous" nonlinear Klein-Gordon equation, no small time periodic solution exists generically (\cite{Sigal93CMP, SW99IM, BC11AJM}).
However, for the discrete nonlinear Klein-Gordon equations, we show that there exist small time periodic solutions.
\end{abstract}

\section{Introduction}
In this note, we consider discrete nonlinear Klein-Gordon equations:
\begin{equation}\label{DNLKG}
u_{tt}+ H u + m^2 u + u^p=0,\quad (t,j)\in \R\times\Z,
\end{equation}
where $u(t,j)$ is a real valued unknown function, $m>0$, $p\in\N$, $p\geq 2$ and $H=-\Delta + V$ with
\begin{itemize}
\item
$\(\Delta u\)(j)=u(j+1)-2u(j)+u(j+1)$ and,
\item
$\(Vu\)(j)=V(j)u(j)$ and $|V(j)|\to 0$ as $|j|\to \infty$.
\end{itemize}

\noindent
Discrete nonlinear Klein-Gordon equation \eqref{DNLKG} can be considered as a discretized model of the "continuous" nonlinear Klein-Gordon equation:
\begin{equation}\label{NLKG}
u_{tt}+ H u + m^2 u + u^p=0,\quad (t,x)\in \R\times\R^d,
\end{equation}
where $d\geq 1$ and $H=-\Delta + V$ with $\Delta$ being the usual Laplacian on $\R^d$.

We now assume that (in both continuous and discrete cases),
\begin{align}
\sigma_d(H)=\{e\},\quad \text{with}\quad m^2+e>0,
\end{align}
where $\sigma_d(H)$ is the set of discrete spectrum of $H$.
We further set $\phi$ to be the normalized eigenfunction of $H$ associated to $e$.
Under such assumption, for the linear Klein-Gordon equation:
\begin{equation}\label{KG}
u_{tt}+Hu + m^2 u =0,
\end{equation}
one can easily show that there exists a family of periodic solutions
\begin{equation*}
\psi(t)=a \phi \sin \omega t,\ \mathrm{where}\ \omega = \sqrt{m^2+e}\ \mathrm{and}\ a\in\R.
\end{equation*}
Notice that such periodic solutions are stable under suitable linear perturbation of the equation (perturbation of $V$).
This fact comes from the stability of eigenvalues of $H$ under small perturbation of $V$ (see for example \cite{KatoPertBook}, \cite{RS4}).

On the other hand, in the continuous case, it was shown by Sigal \cite{Sigal93CMP} that periodic solutions are unstable under generic nonlinear perturbation.
More precisely, under some generic nondegeneracy assumption (which we will explain below), there are no periodic solutions near $\psi$ for the perturbed equation:
\begin{equation*}
u_{tt}+Hu +m^2 u + \epsilon u^p=0,
\end{equation*}
where $\epsilon$ is small.
By a simple scaling argument, we see that Sigal's result corresponds to the nonexistence of small periodic solution of \eqref{NLKG}.
Later, Soffer-Weinstein \cite{SW99IM} and Bambusi-Cuccagna \cite{BC11AJM} improved this result by showing that all small solutions scatter and therefore no periodic solution exists.
More precisely, in the case $d=3$ and $p\geq 3$, they showed that all small solutions behaves like linear solutions of \eqref{KG} with $V\equiv 0$.
Here, Soffer-Weinstein \cite{SW99IM} considered the case $\sigma_d(H)=\{-\frac 3 4 m^2<e<0\}$ and Bambusi-Cuccagna \cite{BC11AJM} generalized it to the case $\sigma_d(H)=\{-m^2<e_1<\cdots<e_n<0\}$.

We now briefly explain the mechanism which prohibits the existence of periodic solutions by following \cite{BC11AJM}.
First, recall that nonlinear Klein-Gordon equation is a Hamiltonian equation.
That is, \eqref{NLKG} can be rewritten as
\begin{align*}
\dot v = -\nabla_u \mathcal H,\quad \dot u =\nabla_v \mathcal H,
\end{align*}
where $(u,v)=(u,u_t)$ and
\begin{align*}
\mathcal H(u,v)=\frac 1 2\int \(v^2 +|\nabla u|^2 + Vu + m^2u^2\)\,dx + \frac{1}{p+1}\int u^{p+1}\,dx.
\end{align*}
We now decompose $(u,v)$ with respect to the spectrum of $H=-\Delta+V$ such as
$$
u=p\phi+P_c u,\quad v=q \phi+ P_cv,\quad\text{where}\quad P_c=1-\<\cdot,\phi\>\phi.$$
Further, we introduce the complex variables 
$$\xi=\frac{q \omega^{1/2}+\im p \omega^{-1/2}}{\sqrt{2}},\quad
 f =\frac{ B^{1/2}P_c u + \im B^{-1/2}P_c v}{\sqrt 2}\quad \text{where}\quad B=\sqrt{H+m^2}.$$
The complex variables now satisfy the system
\begin{align}\label{eq:1}
\im \dot \xi =\partial_{\bar \xi}\mathcal H,\quad \im \dot f = \nabla_{\bar f}\mathcal H.
\end{align}
Our task now becomes to show that $\xi$ decays to $0$ as $t\to \infty$.
In particular, we want to show that the system exhibits some kind of dumping for $\xi$ even though it is a Hamiltonian (conserved) system.
To show this we use Birkhoff normal form argument.
After a canonical change of the coordinate, the Hamiltonian can be reduced to the effective Hamiltonian $\mathcal H_{\mathrm{eff}}$ "$+$error", where
\begin{align*}
\mathcal H_{\mathrm{eff}}=\omega |\xi|^2 + \<Bf,f\>+\<\xi^N G,f\>,
\end{align*}
for some $G\in \mathcal S$ (Schwartz function).
Here, $\<f,g\>=\mathrm{Re}\int f\bar g\,dx$ and $N$ is the smallest natural number satisfying $(N-1)^2 \omega^2<m^2<N^2 \omega^2$.
\begin{remark}
The condition $\sigma_d(H)=\{-\frac 3 4 m^2<e<0\}$ given in Soffer-Weinstein \cite{SW99IM} corresponds to the case $N=2$.
\end{remark}
\noindent
Then, by \eqref{eq:1}, the equations for $\xi$ and $f$ (ignoring the error terms) are given by
\begin{align*}
\im \dot \xi = \omega \xi + N \bar \xi^{N-1}\int \bar G f\,dx,\quad \im \dot f = Bf + \xi^NG.
\end{align*}
Now, if we take $\im \dot \xi\sim \omega \xi$ as a first approximation, we have $f\sim -\xi^N(B-\omega N-\im 0)^{-1}G$.
Then, substituting this again into the equation of $\xi$ and multiply by $\bar \xi$ and taking the imaginary part, we obtain,
\begin{align*}
\frac {d}{dt}|\xi|^2 =-2N \omega^2 \Gamma |\xi|^{2N},\quad \Gamma=\Im \<(B-\omega N-\im 0)^{-1}G,G\>.
\end{align*}
Notice that one can show $\Gamma\geq 0$.
Therefore, if we assume $\Gamma\neq 0$ (which implies $\Gamma>0$), we see that $\xi$ decays, which is the desired result.

The formula of $\Gamma$ is the nonlinear analogue of Fermi Golden Rule which appears in the theory of resonance in quantum mechanics (see, for example \cite{RS4, SW98GAFA}).
In the following, the word Fermi Golden Rule will simply mean for the coefficient $\Gamma$.
The nonvanishing of Fermi Golden Rule ($\Gamma\neq 0$) is assumed for all papers \cite{Sigal93CMP}, \cite{SW99IM} and \cite{BC11AJM}.

Notice that it is crucial to have $\omega N \in \sigma_{\mathrm{cont}}(B)$ (continuous spectrum of $B$), because if $\omega N \not\in \sigma(B)$, then $$\Im \<(B-\omega N-\im 0)^{-1}G,G\>=\Im \<(B-\omega N)^{-1}G,G\>=0,$$
and if $\omega N \in \sigma_d(B)$, we will not be able to define $(B-\omega N -\im 0)^{-1}$.
As a conclusion, the nonvanishing assumption of Fermi Golden Rule is a nondegeneracy condition related to the nonlinear interaction between continuous spectrum and the discrete spectrum.
The nonlinear interaction occurs because the nonlinearity create new frequencies $n^2 \omega^2$ which eventually collides with the continuous spectrum of $H+m^2$ which is $[m^2,\infty)$ in the continuous case.

Now, we come back to the discrete case.
In this case, the continuous spectrum of $H$ is $[0,4]$.
This implies that there is a possibility that the frequencies $n^2 \omega^2$ may "jump over" the continuous spectrum (there are also possibilities that $m^2+e>m^2+4$ from the beginning).
In such cases, one can expect that there are no nonlinear interaction between the point spectrum and the continuous spectrum.
This means that one can expect the existence of periodic solutions of the nonlinear problem \eqref{DNLKG} and this is what we show in this note.

\begin{remark}
The above observation for the discrete case was first given by \cite{Cuccagna10DCDS} for the discrete nonlinear Schr\"odinger equation in the context of the asymptotic stability of small standing waves.
\end{remark}

\begin{assumption}\label{ass:1}
We assume that $\sigma_d(H)=\{e\}$ where $e\not\in [0,4]$ with $\mathrm{dim}\mathrm{ker}(H-e)=1$.
Further, we assume $m^2+e>0$.
We set $\omega=\sqrt{m^2+e}$ and assume
\begin{equation}\label{eq:nonres}
n^2 \omega^2-m^2\not\in [0,4]\quad \mathrm{for\ all}\ n\in\N.
\end{equation}
We set $\phi$ to be the normalized eigenfunction of $e$.
\end{assumption}

\begin{remark}
The condition \eqref{eq:nonres} is the crucial assumption that guarantees there is no nonlinear interaction between the point spectrum and the continuous spectrum.
Notice that this kind of condition is never satisfied in the continuous case because in this case, the r.h.s.\ of \eqref{eq:nonres} has to become $[0,\infty)$.
\end{remark}

\begin{remark}
If we assume $\sum_{j\in \Z} (1+|j|)|V(j)|<\infty$, we have $\mathrm{dim}\mathrm{ker}(H - e)=1$.
See Lemma 5.3 of \cite{CT09SIAM}.
\end{remark}

In this note, we use the following notations.
\begin{itemize}
\item
$\<u,v\>=\sum_{j \in \Z}u(j)v(j)$.
\item
$l_e^a:=\{u=\{u(j)\}_{j\in\Z}\ |\ \|u\|_{l_e^a}:=\(\sum_{j\in\Z}e^{2a|j|}|u(j)|^2\)^{1/2}<\infty\}$, and $\<u,v\>_{l_e^a}:=\<T_a u, T_a v\>$, where $(T_a u)(j):=e^{a|j|}u(j)$.
\item
For Banach spaces $X,Y$, $\mathcal L(X;Y)$ will mean the Banach space of bounded operators from $X$ to $Y$.
If $X=Y$, we set $\mathcal L(X):=\mathcal L(X;X)$.
Further, we inductively set $\mathcal L^n(X;Y)$ by $\mathcal L^n(X;Y):=\mathcal L(X;\mathcal L^{n-1}(X;Y))$ for $n\geq 2$ and $\mathcal L^1(X;Y):=\mathcal L(X;Y)$.
\item
If there exists $C>0$ which is independent of parameters which we are considering and we have $a\leq Cb$, then we write $a\lesssim b$.
\item
For a Banach space $X$ and $r>0$, we set $B_X(r):=\{u\in X\ |\ \|u\|_{X}< r\}$, where $\|\cdot\|_X$ is the norm of $X$.
\item
By $C^\omega(B_X(r),Y)$ for $X=\R,\C$, we mean the set of (real) analytic $Y$-valued functions.
In the case $X=\C$, we mean that $f\in C^\omega(B_{\C}(r),Y)$ is real analytic with respect to $z_R$, $z_I$ where $z=z_R+\im z_I$. 
\item
$\delta_{ij}=1$ if $i=j$ and $\delta_{ij}=0$ if $i\neq j$.
\end{itemize}

\begin{remark}
It is well known that $\phi$ decays exponentially.
Therefore, there exists some $a>0$ s.t.\  $\phi\in l_e^{a}$.
For the convenience of the readers, we have proved this fact in the appendix of this note.
\end{remark}

The main result in this note is the following.

\begin{theorem}\label{thm:1}
Under Assumption \ref{ass:1}, there exist $a>0$ and $\delta_0>0$ s.t.\ there exist 
\begin{align*}
\Phi\in C^\omega(B_\C(\delta_0),l_e^{a})\quad \mathrm{and}\quad E\in C^\omega(B_\R(\delta_0^2),\R),
\end{align*} 
s.t.\ $\Phi[z]$ is a real valued solution of \eqref{DNLKG} if $z$ satisfies 
\begin{align*}
\im z_t = \omega(|z|^2)z,\quad \mathrm{ where}\quad \omega(|z|^2)=\sqrt{E(|z|^2)+m^2}.
\end{align*}
Further, we have
\begin{align}
&\|\Phi[z]-(z+\bar z) \phi\|_{l_e^{a}}\lesssim |z|^p,\label{eq:5}\\&
|E(|z|^2)-e|\lesssim \begin{cases} |z|^{p-1} \quad \mathrm{if}\ p:\mathrm{odd}, \\ |z|^{2p-2} \quad \mathrm{if}\ p:\mathrm{even}.\end{cases}\label{eq:6}
\end{align}
\end{theorem}

We note that Theorems \ref{thm:1} corresponds to Theorems 1.6 of \cite{MDNLS1} for the discrete nonlinear Schr\"odinger equations with potential.
In the (continuous and discrete) nonlinear Schr\"odinger equation case, because of the gauge invariance, one can easily show that there exist small periodic solutions of the form $e^{\im \omega t}\phi(x)$.
However, under the assumption of the nonvanishing of Fermi Golden Rule, no small quasi-periodic solution exists (\cite{Sigal93CMP, TY02ATMP, TY02CPDE, TY02IMRN,SW04RMP,CuMaAPDE}).
On the other hand, due to the boundedness of the continuous spectrum, one can expect that there exist quasi-periodic solutions of two modes and this is shown in \cite{MDNLS1}.

The proof of Theorem \ref{thm:1} is parallel to \cite{MDNLS1} and actually more simple than \cite{MDNLS1} because in this case we are only handling a periodic solution and not a quasi-periodic solution.
However, we would like to present the proof in this note because of its simplicity and moreover we are not aware for similar results of this kind besides \cite{MDNLS1}.

\section{Proof of Theorem \ref{thm:1}}\label{sec:proof1}
In this section, we prove Theorem \ref{thm:1}.
We start from the ansatz:
\begin{equation}\label{eq:10}
\Phi[z]=\(z +\bar z\) \phi + \sum_{n\geq 0} \(z^n +\bar z^n\) v _n,
\end{equation}
with $\<v_1,\phi\>=0$.
We assume that $z$ satisfies 
\begin{equation}\label{eq:11}
\im \dot z = \sqrt{E+m^2} z,
\end{equation}
 where $E=e+\varepsilon$.
Then, our task is to determine $\mathbf v=\{v_m\}_{m\geq 0}$ and $\varepsilon$ for given $z$ to make $\Phi[z]$ to be the solution of \eqref{DNLKG} when $z$ satisfies \eqref{eq:11}.
Assuming \eqref{eq:11}, we have
\begin{align*}
\frac{d^2}{dt^2}\Phi[z]&=-(e+\varepsilon + m^2)\(z+\bar z\)\phi - \sum_{n\geq 1} n^2(e+\varepsilon+m^2)\( z^n + \bar z^n\) v_n,\\
\(H+m^2\)\Phi[z]&= (e+m^2)(z+\bar z)\phi + \sum_{n\geq 0}(z^n+\bar z^n)(H+m^2)v_n,\\
(\Phi[z])^p&=\sum_{n\geq 0} (z^n+\bar z^n)w_n(|z|^2,\mathbf v),
\end{align*}
where the third line is the definition of $w_n(|z|^2,\mathbf v)$.
Therefore, we have the following system:
\begin{align}
&\varepsilon = \< \phi , w_1(|z|^2,\mathbf v)\>,\label{eq:12}\\
&(H-e)v_1 =\varepsilon v_1 -(1-Q) w_1(|z|^2,\mathbf v),\label{eq:13}\\
&\(H-(n^2(e+m^2)-m^2)\)v_n=n^2 \varepsilon v_n - w_n(|z|^2,\mathbf v),\quad n\neq 1,\label{eq:14}
\end{align}
where $Q=\<\cdot,\phi\>\phi$.
We define $\varepsilon(|z|^2,\mathbf v)$ to be the r.h.s.\ of \eqref{eq:12}.
Thus, it remains to determine $\mathbf v$ for given $|z|^2$.
We now set
\begin{align*}
X_{a,r}:=\{ \mathbf v=\{v_n\}_{n\geq 0}\subset l_e^a\ |\ \|\mathbf v\|_{a,r}:=\sum_{n\geq 0}r^n \|v_n\|_{l_e^a}<\infty\},
\end{align*}
and try to reformulate \eqref{eq:13}-\eqref{eq:14} as a fixed point problem (we think \eqref{eq:12} is the definition of $\varepsilon$).
Set
\begin{align}\label{eq:14.1}
\mathcal P \mathbf v :=\{\(1-\delta_{n1}Q \)v_n\}_{n\geq 0},\quad \text{and}\quad X_{a,r}^c:=\mathcal P X_{a,r}.
\end{align}
We next set the operators $\mathcal A$, $\mathcal B$ on $X_{a,r}^c$ as
\begin{align}
&\mathcal A \mathbf v = \{(H-(n^2(e+m^2)-m^2))^{-1}v_n\}_{n\geq 0},\label{eq:14.2}\\&
\mathcal B \mathbf v= \{(H-(n^2(e+m^2)-m^2))^{-1}n^2 v_n\}_{n\geq 0},\label{eq:14.3}
\end{align}
where $\mathbf v=\{v_n\}_{n\geq 0} \in X^c_{a,r}$.

\begin{remark}
By \eqref{eq:nonres} in Assumption \ref{ass:1}, we see that each $(H-(n^2(e+m^2)-m^2))^{-1}(1-\delta_{n1}Q)$ exists.
Notice that if \eqref{eq:nonres} do not hold, we will not be able to invert $H-(n^2(e+m^2)-m^2)$ for some $n\geq 1$ and our strategy completely fails. 
Further by the discussion given in the introduction of this note, we expect there exists no periodic solution.
\end{remark}

To express $\{w_n(|z|^2,\mathbf v)\}_{n\geq 0}$, we introduce the following multilinear operator on $X_{a,r}$.

\begin{definition}
Let $\mathbf v_k=\{v_{k,n}\}_{n\geq 0}$ for $k=1,2$.
We define $\mathcal M(|z|^2,\mathbf v_1,\mathbf v_2)=\{M_n(|z|^2,\mathbf v_1,\mathbf v_2)\}_{n\geq 0}$ by
\begin{align}\label{eq:14.30}
\sum_{n\geq 0} (z^n+\bar z^n)M_n(|z|^2,\mathbf v_1,\mathbf v_2)=\sum_{n_1\geq 0} (z^{n_1}+\bar z^{n_1})v_{1,n_1}\sum_{n_2\geq 0} (z^{n_2}+\bar z^{n_2})v_{2,n_2}.
\end{align}
We inductively define $\mathcal M_k(|z|^2,\mathbf v_1,\cdots,\mathbf v_k)$ by 
\begin{align*}
\mathcal M_2(|z|^2,\mathbf v_1,\mathbf v_2)&=\mathcal M(|z|^2,\mathbf v_1,\mathbf v_2)\quad \mathrm{and}\\ \mathcal M_k(|z|^2,\mathbf v_1,\cdots,\mathbf v_k)&=\mathcal M_{2}(|z|^2,\mathbf v_1,\mathcal M_{k-1}(|z|^2,\mathbf v_2,\cdots,\mathbf v_k)),
\end{align*}
and set $\mathcal M_k(|z|^2,\mathbf v):=\mathcal M_k(|z|^2,\mathbf v,\cdots,\mathbf v)$.
\end{definition}
Let 
\begin{align}\label{eq:14.31}
\Phi_0:=\{\delta_{n1}\phi\}_{n\geq 0}.
\end{align}
Then, we can express $\mathcal W(|z|^2,\mathbf v):=\{w_n(|z|^2,\mathbf v)\}_{n\geq 0}$ as
\begin{align}\label{eq:14.4}
\mathcal W(|z|^2,\mathbf v)=\mathcal M_p(|z|^2,\Phi_0+\mathbf v).
\end{align}
Using \eqref{eq:12}, \eqref{eq:14.1}, \eqref{eq:14.2}, \eqref{eq:14.3}, \eqref{eq:14.31} and \eqref{eq:14.4}, we define
\begin{align}\label{eq:14.5}
\mathbf \Phi(|z|^2,\mathbf v):=\varepsilon(|z|^2,\mathbf v)\mathcal B \mathbf v +\mathcal A \mathcal P \mathcal M_p(|z|^2,\Phi_0+\mathbf v).
\end{align}
Then, we can reformulate the system \eqref{eq:13}-\eqref{eq:14} as a fixed point problem
\begin{align}\label{eq:14.6}
\mathbf v=\mathbf \Phi(|z|^2,\mathbf v).
\end{align}
In the following, we show that $\mathbf \Phi$ is well defined and it is a contraction mapping in a small ball of $X_{a,r}^c$ provided $r>0$ sufficiently small.
\begin{lemma}\label{lem:1}
$\mathcal A, \mathcal B\in \mathcal L(X_{a,r}^c)$.
\end{lemma}

\begin{proof}
First, we have $$\| \(H-(n^2(e+m^2)-m^2)\)^{-1}(1-\delta_{n1}Q)\|_{\mathcal L(l_e^a)}\lesssim (1+n)^{-2}, $$
see, for example, Lemma A.1 of \cite{MDNLS1}.
Therefore, $\|(\mathcal A \mathbf v)_n\|_{l_e^a}\lesssim (1+n)^{-2} \|v_n\|_{l_e^a}$ and $\|(\mathcal B v)_n\|_{l_e^a}\lesssim \|v_n\|_{l_e^a}$.
Thus, we have the conclusion.
\end{proof}

\begin{lemma}\label{lem:1.1}
Let $\delta<r$.
Then, we have 
\begin{align*}
\mathcal M_p\in C^\omega(B_\R(\delta^2);\mathcal L^p(X_{a,r};X_{a,r}))\quad \text{with}\quad \sup_{z\in B_\C(\delta)}\|\mathcal M_p(|z|^2,\cdots)\|_{\mathcal L^p(X_{a,r};X_{a,r})}\lesssim 1.
\end{align*}
\end{lemma}

\begin{proof}
We only show the claim of the lemma for the case $p=2$.
For the cases $p\geq 3$, we can easily show the clam from the inductive definition of $\mathcal M_p$.
By \eqref{eq:14.30}, for $\mathbf v_k=\{v_{k,n}\}_{n\geq 0}$ ,
\begin{align*}
&\sum_{n_1\geq 0} (z^{n_1}+\bar z^{n_1})v_{1,n_1}\sum_{n_2\geq 0} (z^{n_2}+\bar z^{n_2})v_{2,n_2}\\&=
\sum_{n_1,n_2\geq 0}(z^{n_1+n_2}+\bar z^{n_1+n_2})v_{1,n_1}v_{2,n_2}+\sum_{n_1,n_2\geq 0}(z^{n_1}\bar z^{n_2}+\bar z^{n_1}z^{n_2})v_{1,n_1}v_{2,n_2}\\&=
\sum_{n\geq 0}(z^n+\bar z^n)\(\sum_{n\geq n_1\geq 0}v_{1,n_1}v_{2,n-n_1}+\sum_{n_2\geq 0}|z|^{2n_2}v_{1,n+n_2}v_{2,n_2}+\sum_{n_1>0}|z|^{2n_1}v_{1,n_1}v_{2,n+n_1}\).
\end{align*}
Therefore, we have
\begin{align*}
&\mathcal M(|z|^2,\mathbf v_1,\mathbf v_2)=\sum_{m\geq 0}|z|^2 \mathbf m_m(\mathbf v_1,\mathbf v_2),\quad \text{where}\\
&\mathbf m_0(\mathbf v_1,\mathbf v_2)=\{m_{0,n}(\mathbf v_1,\mathbf v_2)\}_{n\geq 0}=\{v_{1,n}v_{2,0}+\sum_{n\geq n_1\geq 0}v_{1,n_1}v_{2,n-n_1}\}_{n\geq 0},\\
&\mathbf m_m(\mathbf v_1,\mathbf v_2)=\{m_{m,n}(\mathbf v_1,\mathbf v_2)\}_{n\geq 0}=\{v_{1,n+m}v_{2,m}+v_{1,m}v_{2,n+m}\}_{n\geq 0},\quad m\geq 1.
\end{align*}
Now,
\begin{align*}
\|\mathbf m_0(\mathbf v_1,\mathbf v_2)\|_{a,r}& \leq \sum_{n\geq 0} r^n\(\|v_{1,n}\|_{l_e^a}\|v_{2,0}\|_{l_e^a}+\sum_{n\geq n_1\geq 0} \|v_{1,n_1}\|_{l_e^a}\|v_{2,n-n_1}\|_{l_e^a}\)\\&\leq 2\|\mathbf v_1\|_{a,r}\|\mathbf v_2\|_{a,r},
\end{align*}
and
\begin{align*}
\|\mathbf m_m(\mathbf v_1,\mathbf v_2)\|_{a,r}& \leq r^{-2m}\sum_{n\geq 0} r^{n+2m}\(\|v_{1,n+m}\|_{l_e^a}\|v_{2,m}\|_{l_e^a}+\|v_{1,m}\|_{l_e^a}\|v_{2,n+m}\|_{l_e^a}\)\\&\leq 2 r^{-2m}\|\mathbf v_1\|_{a,r}\|\mathbf v_2 \|_{a,r}.
\end{align*}
Thus, we have $\mathbf m_m\in \mathcal L^2(X_{a,r};X_{a,r})$.
Further, if $|z|^2<r^2$, $\sum_{m\geq 0}|z|^{2m} \|\mathbf m_m \|_{\mathcal L^2(X_{a,r};X_{a,r})}$ converges.
Therefore, we have $\mathcal M \in C^\omega(B_\R(\delta^2);\mathcal L^2(X_{a,r};X_{a,r}))$.
\end{proof}

We define $\mathcal C$ by
\begin{align*}
\mathcal C(\mathbf v):=\<\phi,v_1\>,\quad \text{where}\quad \mathbf v=\{v_n\}_{n\geq 0}.
\end{align*}
Since $|\mathcal C(\mathbf v)|\leq \|\phi\|_{l^2} \|v_1\|_{l^2}\leq r^{-1}\|\phi\|_{l^2} \|\mathbf v\|_{a,r}$, we see $\mathcal C\in \mathcal L(X_{a,r};\R)$.
Using, $\mathcal C$ and $\mathcal M_p$, we can express $\varepsilon$ by
\begin{align}\label{15}
\varepsilon(|z|^2,\mathbf v)=\mathcal C\circ \mathcal M_p(|z|^2,\Phi_0+\mathbf v).
\end{align}
Therefore, we have 
\begin{align}\label{16}
\varepsilon\in C^\omega(B_{\R}(\delta^2)\times X_{a,r};\R)),\quad\text{ with}\quad |\varepsilon(|z|^2,\mathbf v)|\lesssim r^{-1}(r^p+\|\mathbf v\|_{a,r}^p).
\end{align}

\begin{remark}\label{rem:epsilon}
If $p$ is even, then $\varepsilon(|z|^2,0)=0$ because $(z+\bar z)^p$ has no term such as $z|z|^{2m}$.
\end{remark}
%
%

\begin{proof}[Proof of Theorem \ref{thm:1}]
Let $\delta< r$.
By Lemma \ref{lem:1.1} and \eqref{16}, we have $\mathbf \Phi \in C^\omega(B_\R(\delta^2)\times X_{a,r}^c;X_{a,r}^c)$. 
Further, for $\mathbf v_1,\mathbf v_2\in X_{a,r}$ and $z\in B_\C(\delta)$,
\begin{align*}
\|\mathbf \Phi(|z|^2,\mathbf v_1)\|_{a,r}&\lesssim \(r^{-1}\|\mathbf v_1\|_{a,r}+1\)(r+\|\mathbf v_1\|)^p,\\
\|\mathbf \Phi(|z|^2,\mathbf v)-\mathbf \Phi(|z|^2,\mathbf v_2)\|_{a,r}&\lesssim \(r^{-1}\(\|\mathbf v_1\|_{a,r}+\|\mathbf v_2\|_{a,r}\)+1\)(r+\|\mathbf v_1\|+\|\mathbf v_2\|_{a,r})^{p-1}\|\mathbf v_1-\mathbf v_2\|_{a,r}.
\end{align*}
Therefore, we see that $\mathbf \Phi(|z|^2,\cdot)$ is a contraction mapping on $B_{X_{a,r}^c}(Cr^p)$ for some $C>0$ and $r\ll1$.
By contraction mapping theorem, we have $\mathbf v\in C^\omega(B_{\R}(\delta^2);X_{a,r}^c)$ s.t.\ $\mathbf v(|z|^2)$ satisfies \eqref{eq:14.6}.

We have Theorem \ref{thm:1} because $\Phi[z]=(z+\bar z)\phi+\sum_{n\geq 0}(z^n+\bar z^n)v_n(|z|^2)$ is a solution of \eqref{DNLKG}.
Notice that the estimate \eqref{eq:6} follows from Remark \ref{rem:epsilon} above.
\end{proof}

\appendix

\section{Exponential decay of eigenfunction of $H$}

This appendix is devoted for the proof of the following proposition.

\begin{proposition}\label{prop:1}
Let $e\in \R\setminus[0,4]$ be the eigenvalue of $H=-\Delta + V$ and let $\phi\in l^2$ be the normalized eigenfunction of $H$ associated $e$.
Further, assume $\mathrm{dim}\mathrm{ker}(H -e)=1$.
Then, there exists $a>0$ s.t.\ $\phi \in l_e^a$.
\end{proposition}

Set $T_a: l_e^a \to l^2$ by
\begin{align*}
\(T_a u\)(j):= e^{a|j|} u(j).
\end{align*}

Before the proof of Proposition \ref{prop:1}, we claim that the following fact holds.
\begin{lemma}\label{lem:a1}
Under the assumptions of Proposition \ref{prop:1}, for sufficiently small $a>0$ there exist $e_a\in \R$ and $\phi_a\in l^2(\Z)\setminus\{0\}$ s.t.\ $e_a\to e$ as $a\to 0$ and
\begin{align}\label{2}
T_a(-\Delta + V) T_a^{-1} \phi_a = e_a \phi_a.
\end{align}
\end{lemma}

Proposition \ref{prop:1} is a direct consequence of Lemma \ref{lem:a1}
\begin{proof}[Proof of Proposition \ref{prop:1} assuming Lemma \ref{lem:a1}]
Since $\phi_a$ satisfies \eqref{2}, we have
\begin{align*}
HT_a^{-1} \phi_a = e_a T_{a}^{-1}\phi_a.
\end{align*}
Since $\phi_a\in l^2$, $T_a^{-1}\phi_a\in l_e^a\hookrightarrow l^2$.
Therefore, $e_a$ is an eigenvalue of $H$.
However, since $e_a\to e$ as $a\to 0$ and $e$ is an isolated eigenvalue, $e_a$ has to be equal to $e$ if $a>0$ is sufficiently small.
Therefore, since $\mathrm{dim}\mathrm{ker}(H-e)=1$, there exists $c_a\in \R$ s.t.
\begin{align*}
\phi=c_a T_a^{-1}\phi_a.
\end{align*}
Finally,
\begin{align*}
\|\phi\|_{l_e^a}=|c_a|\|T_a^{-1}\phi_a\|_{l_e^a}=|c_a|\|\phi_a\|_{l^2}<\infty.
\end{align*}
Therefore, we have the conclusion.
\end{proof}

Before proving Lemma \ref{lem:a1}, we show that $T_a H T_a^{-1}$ is a small bounded perturbation of $H$.

\begin{lemma}\label{lem:2}
There exists a bounded operator $B_a:l^2\to l^2$ s.t. $\|B_a\|_{l^2\to l^2}\lesssim a$ s.t.
\begin{align*}
T_a H T_a^{-1}= H + B_a.
\end{align*}
\end{lemma}

\begin{proof}
By,
\begin{align*}
T_aH T_a^{-1} = H +T_a (-\Delta)T_a^{-1}+\Delta,
\end{align*}
it suffices to show
\begin{align*}
\|T_a (-\Delta)T_a^{-1}+\Delta\|_{l^2\to l^2}\lesssim a.
\end{align*}
Indeed, we have
\begin{align*}
\(\(T_a (-\Delta)T_a^{-1}+\Delta\)u\)(j)=\(1-e^{a(|j|-|j+1|)}\)u(j+1)+\(1-e^{a(|j|-|j-1|)}\)u(j-1),
\end{align*}
and
\begin{align*}
|1-e^{a(|j|-|j+1|)}|+|1-e^{a(|j|-|j-1|)}|\lesssim a.
\end{align*}
Therefore,
\begin{align*}
\|\(T_a (-\Delta)T_a^{-1}+\Delta\)u\|_{l^2}\lesssim a\|u\|_{l^2}
\end{align*}
and we have the conclusion.
\end{proof}

We now prove Lemma \ref{lem:a1}.

\begin{proof}[Proof of Lemma \ref{lem:a1}]
We show that there exist $\delta_a>0$ and $u_a\perp \phi$ (i.e. $\<\phi,u_a\>=0$) s.t.\  $e_a=e+\delta_a$ and $\phi_a=\phi+u_a$ satisfying \eqref{2}. 
Let $B_a$ be the bounded operator given in lemma \ref{lem:2}.
Then, it suffices to find a pair $(e_a,u_a)$ satisfying
\begin{align}\label{3}
(H -e)u_a = \delta_a (\phi + u_a) - B_a (\phi + u_a).
\end{align}
Taking the inner-product between $\eqref{3}$ and $\phi$, we have
\begin{align}\label{4}
\delta_a=\<B_a (\phi + u_a),\phi\>.
\end{align}
Applying $P_c$ to \eqref{3} and taking the inverse of $(H-e)$ (this is invertible in $P_cl^2$, where $P_c$ is the orthogonal projection with respect to $\phi$), we have
\begin{align}\label{5}
u_a= (H -e)^{-1}P_c\(\(\<B_a(\phi+u_a),\phi\>-B_a\)(\phi+u_a)\).
\end{align}

\noindent
So, setting $\Psi(u)$ equal to the r.h.s.\ of \eqref{5},
it suffices to find a fixed point of $\Psi$.
First,
\begin{align*}
\|\Psi(0)\|_{l^2}=\|(H -e)^{-1}P_c\(\(\<B_a\phi,\phi\>-B_a\)\phi\)\|_{l^2}\leq C_0 a,
\end{align*}
for some constant $C_0>0$ independent to $a$.
Second, for $v,w\in l^2$ with  $\|v\|_{l^2},\|w\|_{l^2}\leq 1$, we have
\begin{align*}
\|\Psi(v)-\Psi(w)\|_{l^2}&=\|(H-e)^{-1}P_c\(\<B_a(v-w),\phi\>(\phi+v) +\(\<B_a(\phi+w),\phi\>-B_a\)(v-w) \)\|_{l^2}\\&
\leq C_1a \|v-w\|_{l^2},
\end{align*}
for some constant $C_1>0$ independent of $a$, $v$ and $w$.
Therefore, if $\max(C_0,C_1)a<1/2$, for $u,v$ with $\|u\|_{l^2}, \|v\|_{l^2}<2C_0a$, we have 
\begin{align*}
\|\Psi(u)\|_{l^2}\leq \|\Psi(u)-\Psi(0)\|_{l^2}+\|\Psi(0)\|_{l^2}\leq C_1a\cdot 2C_0a+C_0a<2C_0a,
\end{align*}
and $\|\Psi(u)-\Psi(v)\|_{l^2}\leq \frac 1 2 \|u-v\|_{l^2}$.
This implies that $\Psi$ is a contraction mapping on $B_{l^2}(0,2C_0a)=\{u\in l^2\ |\ \|u\|_{l^2}\leq 2C_0a$.
Therefore, we have the conclusion.
\end{proof}

\subsection*{Acknowledgments}   
The author was supported by the Japan Society for the Promotion of Science (JSPS) with the Grant-in-Aid for Young Scientists (B) 15K17568.

%

Department of Mathematics and Informatics,
Faculty of Science,
Chiba University,
Chiba 263-8522, Japan

{\it E-mail Address}: {\tt maeda@math.s.chiba-u.ac.jp}

\end{document}